\RequirePackage[l2tabu, orthodox]{nag} 

\documentclass{amsart}

\pdfoutput=1

\usepackage[utf8]{inputenc} 
\usepackage[colorlinks, citecolor=red, breaklinks=true]{hyperref} 
\usepackage[stretch=10, shrink=10]{microtype} 
\usepackage{todonotes} 
\usepackage{cases}
\usepackage{amscd, amssymb}  
\usepackage{commath} 
\usepackage{mathtools} 
\usepackage{doi}
\usepackage{xcolor}
\usepackage[initials,msc-links,non-compressed-cites,nobysame]{amsrefs}[2007/10/22]
\usepackage{subcaption}
\usepackage{paralist,enumitem}
\setlist{listparindent=0pt,parsep=3pt}
\setdefaultenum{1.}{}{}{}

\newcommand{\TitleWithUrl}[1]{\IfEmptyBibField{doi}%
  {\IfEmptyBibField{url}{\textit{#1}}%
    {\IfEmptyBibField{eprint}{\href {\BibField{url}}{\textit{#1}}}{\textit{#1}}}%
    }%
  {\href {https://doi.org/\BibField{doi}}{\textit{#1}}}}
\renewcommand{\eprint}[1]{\IfEmptyBibField{url}{\url{#1}}%
  {\href {\BibField{url}}{#1}}}

\BibSpec{article}{%
    +{}  {\PrintAuthors}                {author}
    +{,} { \TitleWithUrl}               {title}
    +{.} { }                            {part}
    +{:} { \textit}                     {subtitle}
    +{,} { \PrintContributions}         {contribution}
    +{.} { \PrintPartials}              {partial}
    +{,} { }                            {journal}
    +{}  { \textbf}                     {volume}
    +{}  { \PrintDatePV}                {date}
    +{,} { \issuetext}                  {number}
    +{,} { \eprintpages}                {pages}
    +{,} { }                            {status}
    +{,} { available at arXiv:\eprint}        {eprint}
    +{}  { \parenthesize}               {language}
    +{}  { \PrintTranslation}           {translation}
    +{;} { \PrintReprint}               {reprint}
    +{.} { }                            {note}
    +{.} {}                             {transition}
    +{}  {\SentenceSpace \PrintReviews} {review}
}

\BibSpec{collection.article}{%
    +{}  {\PrintAuthors}                {author}
    +{,} { \TitleWithUrl}                     {title}
    +{.} { }                            {part}
    +{:} { \textit}                     {subtitle}
    +{,} { \PrintContributions}         {contribution}
    +{,} { \PrintConference}            {conference}
    +{}  {\PrintBook}                   {book}
    +{,} { }                            {booktitle}
    +{,} { \PrintEditorsA}              {editor}
    +{,} { }                            {publisher}
    +{,} { }                            {organization}
    +{,} { }                            {address}    
    +{,} { \PrintDateB}                 {date}
    +{,} { pp.~}                        {pages}
    +{,} { }                            {status}
    +{,} { available at \eprint}        {eprint}
    +{}  { \parenthesize}               {language}
    +{}  { \PrintTranslation}           {translation}
    +{;} { \PrintReprint}               {reprint}
    +{.} { }                            {note}
    +{.} {}                             {transition}
    +{}  {\SentenceSpace \PrintReviews} {review}
}
\BibSpec{book}{%
    +{}  {\PrintPrimary}                {transition}
    +{,} { \TitleWithUrl}               {title}
    +{.} { }                            {part}
    +{:} { \textit}                     {subtitle}
    +{,} { \PrintEdition}               {edition}
    +{}  { \PrintEditorsB}              {editor}
    +{,} { \PrintTranslatorsC}          {translator}
    +{,} { \PrintContributions}         {contribution}
    +{,} { }                            {series}
    +{,} { \voltext}                    {volume}
    +{,} { }                            {publisher}
    +{,} { }                            {organization}
    +{,} { }                            {address}
    +{,} { \PrintDateB}                 {date}
    +{,} { }                            {status}
    +{}  { \parenthesize}               {language}
    +{}  { \PrintTranslation}           {translation}
    +{;} { \PrintReprint}               {reprint}
    +{.} { }                            {note}
    +{.} {}                             {transition}
    +{}  {\SentenceSpace \PrintReviews} {review}
}

\BibSpec{thesis}{%
    +{}  {\PrintAuthors}                {author}
    +{.} { \PrintDate}                  {date}
    +{.} { \TitleWithUrl}               {title}
    +{:} { \textit}                     {subtitle}
    +{,} { \PrintThesisType}            {type}
    +{,} { }                            {organization}
    +{,} { }                            {address}
    +{,} { \eprint}                     {eprint}
    +{,} { }                            {status}
    +{}  { \parenthesize}               {language}
    +{}  { \PrintTranslation}           {translation}
    +{;} { \PrintReprint}               {reprint}
    +{.} { }                            {note}
    +{.} {}                             {transition}
    +{}  {\SentenceSpace \PrintReviews} {review}
}

\newtheorem{theo}{Theorem}[section]
\newtheorem{fact}[theo]{Fact}
\newtheorem{prop}[theo]{Proposition}
\newtheorem{lemm}[theo]{Lemma}
\newtheorem{coro}[theo]{Corollary}

\theoremstyle{definition}
\newtheorem{defi}[theo]{Definition}
\newtheorem{exam}[theo]{Example}

\theoremstyle{remark}
\newtheorem{rema}[theo]{Remark}

\numberwithin{equation}{section}

\newcommand{\cratio}{\operatorname{cr}}
\renewcommand{\Re}{\operatorname{Re}}
\newcommand{\SO}{\operatorname{SO(3)}}
\newcommand{\lieo}{\operatorname{o(3)}}
\newcommand{\resp}[1]{\textup{(}resp. #1\textup{)}}

\newsavebox\mybox

\title{Discrete minimal nets with symmetries}

\author{Joseph Cho}
\address[Joseph Cho]{Institute of Discrete Mathematics and Geometry, TU Wien, Wiedner Hauptstrasse 8-10/104, 1040 Wien, Austria}
\email{jcho@geometrie.tuwien.ac.at}
 
\author{Wayne Rossman}
\address[Wayne Rossman]{Department of Mathematics, Graduate School of Science, Kobe University, 1-1 Rokkodai-cho, Nada-ku, Kobe 657-8501, Japan}
\email{wayne@math.kobe-u.ac.jp}

\author{Seong-Deog Yang}
\address[Seong-Deog Yang]{Department of Mathematics, Korea University, 145 Anam-ro, Seoul 02841, Korea}
\email{sdyang@korea.ac.kr}

\date{}
\makeatletter
\@namedef{subjclassname@2020}{\textup{2020} Mathematics
  Subject Classification}
\makeatother

\begin{document}

\begin{abstract}
	In this paper, we extend the notion of Schwarz reflection principle for smooth minimal surfaces to the discrete analogues for minimal surfaces, and use it to create global examples of discrete minimal nets with high degree of symmetry.
\end{abstract}

\maketitle

\section{Introduction}

In the case of smooth minimal surfaces in Euclidean $3$-space $\mathbb{R}^3$, the Schwarz reflection principle has been used to good effect to extend minimal surfaces and study their global behavior.
The Schwarz reflection principle for minimal surfaces comes in two forms.
One states that if the minimal surface lies to one side of a plane and has a curvature-line boundary lying in that plane and meeting it perpendicularly, then the surface extends smoothly by reflection to the other side of the plane.
The other states that if the minimal surface contains a boundary line segment, then it can be smoothly extended across the line by including the 180 degree rotation of the surface about that line.
When one of these two situations holds on a minimal surface, the other one holds on the conjugate minimal surface.

By the nature of the Schwarz reflection principle, we expect that the surfaces constructed will have relatively high degrees of symmetry.
Such symmetry has been seen in numerous works, see, for example, \cite{costa_example_1984, hoffman_complete_1985, hoffman_embedded_1990, karcher_embedded_1988, karcher_triply_1989, karcher_construction_1996, rossman_minimal_1995, rossman_irreducible_1997, schoen_infinite_1970}.

Such symmetry has also been exploited in the discrete case as well: for discrete $S$-isothermic minimal nets, see, for example, \cite{bucking_approximation_2007, bucking_minimal_2008, bobenko_discrete_2017}; for discrete isothermic constant mean curvature nets, see, for example, \cite{hoffmann_discrete_2000}.

In this paper, we investigate how a similar reflection principle will work in the case of discrete isothermic minimal nets and discrete asymptotic minimal nets.
The benefit of this is that it provides us a further tool for extending discrete minimal surfaces described locally (which has been well investigated) to surfaces considered at a more global level (which has not received as much attention yet).
For example, we will construct the central part of a discrete minimal trinoid, which can then be regarded as existing on a global level, since it is not a simply connected surface, as it is topologically equivalent to the sphere minus three disks.
Like in the smooth case, we expect to see relatively high degrees of symmetry in the surfaces we construct in this way.

Our primary results are Proposition \ref{prop:isothermicReflection} and Theorem \ref{theo:asymptoticReflection}, which are the two forms of Schwarz reflection in the discrete case.
In Corollary \ref{coro:linetoplane}, we also show that the two forms of Schwarz reflection are related by conjugate discrete minimal nets, as defined in \cite{hoffmann_discrete_2017}.
Finally, we use these results to produce examples in Section \ref{sect:examples}.

\section{Preliminaries}
Let our domain be a $\mathbb{Z}^2$ lattice with $(m,n) \in \mathbb{Z}^2$, and let $(ijkl)$ denote the vertices of an elementary quadrilateral $((m,n), (m+1, n), (m+1,n+1), (m,n+1))$.
For simplicity, we have chosen our domain to be $\mathbb{Z}^2$; however, the theory will hold true for subdomains of $\mathbb{Z}^2$.
If $F$ is a discrete net $F : \mathbb{Z}^2 \to \mathbb{R}^3$, then we write $F(m,n) = F_{m,n} = F_i$ over any elementary quadrilateral, and let
	\[
		\dif F_{ij} := F_j - F_i.
	\]
A discrete net $F$ is called a \emph{circular net} if $F_i$, $F_j$, $F_k$, and $F_l$ are concircular, representing a discrete notion of curvature line coordinates \cite{nutbourne_differential_1988}.

\subsection{Discrete isothermic nets}
First we recall from \cite[Definition 4]{bobenko_discrete_1996-1} how the cross ratio of four points in $\mathbb{R}^3$ are defined.
\begin{defi}
	Let $x_1, \ldots, x_4 \in \mathbb{R}^3$, and let $\mathbb{R}^3$ be identified with the set of quaternions $\mathbb{H}$ under the usual identification $\mathbb{R}^3 \ni x_i \sim X_i \in \mathbb{H}$.
	The pair of eigenvalues $\{q, \bar{q}\}$ of the quaternion
	\[
		(X_1 - X_2) (X_2 - X_3)^{-1} (X_3 - X_4) (X_4 - X_1)^{-1}
	\]
	is called the \emph{cross ratio of $x_1, \ldots, x_4$}. 
	In the case where $x_1, \ldots, x_4$ are concircular, $q = \bar{q} \in \mathbb{R}$, and we write
	\[
		\cratio(x_1, x_2, x_3, x_4) = q.
	\]
\end{defi}
\begin{rema}
It was further proved in \cite[Lemma 1]{bobenko_discrete_1996-1} that this cross ratio is invariant under Möbius transformations.
\end{rema}
Using this definition of cross ratios, discrete isothermic nets are defined as follows in \cite[Definition 6]{bobenko_discrete_1996-1}:
\begin{defi}
A circular net $F$ is called a \emph{discrete isothermic net} if on every elementary quadrilateral $(ijkl)$,
	\[
		\cratio(F_i, F_j, F_k, F_l) = \frac{a_{ij}}{a_{il}} \in \mathbb{R}_{<0},
	\]
where $a_{ij}$ (resp.\ $a_{il}$) are edge-labeling scalar functions defined on unoriented edges; that is,
	\begin{equation}\label{eqn:edgelabeling}
		a_{ij} = a_{l k} \quad\text{and}\quad a_{il} = a_{jk}
	\end{equation}
on every elementary quadrilateral $(ijkl)$. We call $a_{ij}$ and $a_{il}$ the \emph{cross ratio factorizing functions}.
\end{defi}

It is shown in \cite[Theorem 6]{bobenko_discrete_1996-1} that, for any discrete net $F$, the discrete isothermicity of $F$ is equivalent to the existence of another discrete net $F^*$ such that
	\[
		\dif F^*_{ij} = \frac{a_{ij}}{\| \dif F_{ij}\|^2} \dif F_{ij}, \quad \dif F^*_{il} = \frac{a_{il}}{\| \dif F_{il}\|^2} \dif F_{il}.
	\]
If such an $F^*$ exists, $F^*$ is called a \emph{Christoffel transformation} of $F$, and $(F^*)^* = F$ up to scaling and translation in $\mathbb{R}^3$.

\subsection{Discrete Gaussian and mean curvatures}
For any two parallel circular nets $F$ and $G$, i.e.\ $F$ and $G$ are both circular nets with parallel corresponding edges, the \text{mixed area} of $F$ and $G$ is defined on every elementary quadrilateral as
	\[
		A(F,G)_{ijkl} := \frac{1}{4}(\delta F_{ik} \wedge \delta G_{jl} + \delta G_{ik} \wedge \delta F_{jl})
	\]
where $\delta F_{ik} := F_k - F_i$ and the exterior algebra $\wedge^2 \mathbb{R}^3 (\ni u \wedge v)$ is identified with the Lie algebra $\lieo$, i.e.\ for any $u, v, w \in \mathbb{R}^3$,
	\[
		(u \wedge v)\, w = (u \cdot w) v - (v \cdot w) u
	\]
for the usual inner product of $x, y \in \mathbb{R}^3$ expressed as $x \cdot y$.
Note that $A(F)_{ijkl} := A(F,F)_{ijkl}$ gives the area of the quadrilateral spanned by the image of $F$ over an elementary quadrilateral $(ijkl)$.

It is known through \cite{konopelchenko_three-dimensional_1998} that any circular net $F$ has a parallel circular net $N : \mathbb{Z}^2 \to S^2 \subset \mathbb{R}^3$ taking values in the unit sphere.
Such an $N$ is called a \emph{discrete Gauss map} of $F$.

\begin{rema}\label{rema:normalBundle}
If a discrete line bundle $L : \mathbb{Z}^2 \to \{\text{lines in $\mathbb{R}^3$}\}$ is the normal bundle of $F$, i.e.\ $F_i, F_i + N_i \in L_i$, then $L$ constitutes a discrete line congruence in the sense of \cite[Definition 2.1]{doliwa_transformations_2000}, as any two neighboring lines intersect.
One can see that after a choice of one normal direction at one vertex of $F$ (an initial condition), the line congruence condition and the parallel mesh condition uniquely determine the normal bundle $L$ over all vertices in the domain, since any two neighboring normal lines must intersect at equal distance from the vertices on the surface.
\end{rema}

Furthermore, it is not difficult to see that the parallel net $F^t$ defined as $F^t := F + t N$ for some constant $t$ is also a circular net parallel to $F$.
This allows us to consider the mixed area of $F$ and $F^t$, and recover the discrete version of the Steiner's formula based on mixed areas (see \cite{schief_unification_2003, pottmann_geometry_2007}):
	\begin{align*}
		A(F^t)_{ijkl} &= A(F)_{ijkl} + 2t A(F, N)_{ijkl} + t^2 A(N)_{ijkl} \\
			&= (1 - 2tH_{ijkl} + t^2 K_{ijkl}) A(F)_{ijkl}
	\end{align*}
where $H_{ijkl}$ and $K_{ijkl}$ are defined on each elementary quadrilateral as:
\begin{defi}\label{def:curvatures}
	We call
	\[
		H_{ijkl} = -\frac{A(F,N)_{ijkl}}{A(F)_{ijkl}}, \quad K_{ijkl} = \frac{A(N)_{ijkl}}{A(F)_{ijkl}}
	\]
	the mean and Gaussian curvatures of a circular net $F$ with Gauss map $N$.
\end{defi}

With the notion of mean curvature on any elementary quadrilateral $(ijkl)$ available, discrete isothermic minimal nets and discrete isothermic constant mean curvature (cmc) nets can be defined as:
\begin{defi}\label{def:minimality}
	A circular net $F$ is called a discrete isothermic minimal (resp. cmc) net if $H \equiv 0$ (resp.\ $H \equiv c \neq 0$ for some non-zero constant $c$) on every elementary quadrilateral.
\end{defi}

\subsection{Planar reflection principle for discrete isothermic minimal and cmc nets}
Since circular nets are a discrete analogue of curvature line coordinates, the following notion is natural.
\begin{defi}
	Let $F : \mathbb{Z}^2 \to \mathbb{R}^3$ be a circular net. A discrete space curve $F_{m, n_0}$ (resp. $F_{m_0, n}$) depending on $m$ (resp. $n$) for each $n_0 \in \mathbb{Z}$ (resp. $m_0 \in \mathbb{Z}$) is called a \emph{discrete curvature line}.
\end{defi}
Without loss of generality, let $n_0 \in \mathbb{Z}$, and let $F$ be a discrete isothermic minimal or cmc net, defined on the domain $D := \{(m, n) \in \mathbb{Z}^2 : n \leq n_0\}$ with corresponding Gauss map $N$. Suppose that the discrete curvature line $F_{m, n_0}$ is contained in a plane $\mathcal{P}$, and further suppose that the unit normal at each vertex $(m,n_0)$ is contained in the plane containing the discrete curvature line, i.e.\ $F_{m,n_0} + N_{m,n_0} \in \mathcal{P}$.

If we extend $F$ to the domain $\tilde{D} := \{(m, n) \in \mathbb{Z}^2 : n > n_0\}$ by reflecting the vertices across the plane $\mathcal{P}$, then as mentioned in Remark \ref{rema:normalBundle}, the unit normal $N$ also gets uniquely determined on the extended domain.
The uniqueness of the unit normal and the symmetry of the discrete net then forces the unit normal to be symmetric with respect to $\mathcal{P}$ as well, giving us the following reflective property of minimal and cmc nets:

\begin{prop}\label{prop:isothermicReflection}
	Let $F : D \to \mathbb{R}^3$ be a discrete isothermic minimal (resp.\ cmc) net with corresponding Gauss map $N$.
	Suppose that the discrete curvature line $F_{m, n_0}$ and the normal line congruence $L_{m, n_0}$ along this discrete curve lie in a plane $\mathcal{P}$.
	Extending $F$ to $\mathbb{Z}^2 = D \cup \tilde{D}$ so that the extension is symmetric with respect to $\mathcal{P}$ results in a discrete minimal (resp.\ cmc) net on $\mathbb{Z}^2$.
\end{prop}

\section{Reflection properties of discrete minimal nets}

In this section, we take a closer look at the reflection properties of discrete minimal nets.

\subsection{Discrete isothermic minimal nets}
Exploiting the relationship between holomorphic functions on the complex plane and conformality, a definition of discrete holomorphic functions was given in \cite[Definition 8]{bobenko_discrete_1996-1} as:
\begin{defi}
	A map $g: \mathbb{Z}^2 \to \mathbb{R}^2 \cong \mathbb{C}$ is called a \emph{discrete holomorphic function} if
		\[
			\cratio(g_i, g_j, g_k, g_l) = \frac{a_{ij}}{a_{il}} \in \mathbb{R}_{<0}
		\]
	for some edge-labeling scalar functions $a_{ij}$ and $a_{il}$, i.e.\ satisfying the condition \eqref{eqn:edgelabeling}.
\end{defi}
Using the facts that 
\begin{itemize}
	\item cross ratios are invariant under Möbius transformations,
	\item a discrete isothermic net on the unit sphere corresponds to a discrete holomorphic function on the complex plane via stereographic projection,
	\item the Christoffel transform of a discrete minimal net is its own Gauss map, and
	\item the Christoffel transformation is involutive,
\end{itemize}
a Weierstrass representation for a discrete minimal net was given in \cite[Theorem 9]{bobenko_discrete_1996-1} as follows:

\begin{fact}\label{fact:Weierstrass}
	For a discrete holomorphic function $g$ with cross ratio factorizing functions $a_{ij}$ and $a_{il}$, a discrete isothermic net $F$ defined via
	\[\left\{
		\begin{aligned}
			\dif F_{ij} &= a_{ij}\Re\left((1 - g_i g_j, \sqrt{-1}(1 + g_i g_j), g_i + g_j) \frac{1}{\dif g_{ij}}\right) \\
			\dif F_{il} &= a_{il}\Re\left((1 - g_i g_l, \sqrt{-1}(1 + g_i g_l), g_i + g_l) \frac{1}{\dif g_{il}}\right)
		\end{aligned}
	\right.\]
	becomes a discrete isothermic minimal net.
	Furthermore, any discrete isothermic minimal net can be obtained via some discrete holomorphic function $g$.
\end{fact}

\subsection{Discrete asympotic minimal nets}
In this section, we make use of shift notations:
	\[
		F = F_{m,n}, \: F_1 = F_{m+1,n}, \: F_{\bar{1}} = F_{m-1,n}, \: F_2 = F_{m,n+1}, \: F_{\bar{2}} = F_{m, n-1}.
	\]
Discrete asymptotic nets were defined as follows in several different contexts (see, for example \cite{sauer_parallelogrammgitter_1950, wunderlich_zur_1951, sauer_differenzengeometrie_1970, bobenko_discrete_2008}):

\begin{defi}
	A discrete net $\tilde{F} : \mathbb{Z}^2 \to \mathbb{R}^3$ is a \emph{discrete asymptotic net} if each vertex and its neighboring four vertices are coplanar, i.e.\ $\tilde{F}, \tilde{F}_1, \tilde{F}_{\bar{1}}, \tilde{F}_2, \tilde{F}_{\bar{2}} \in \mathcal{P}_{m,n}$ for some plane $\mathcal{P}_{m,n}$ for each $(m,n)$. 
\end{defi}
Following \cite{bobenko_discrete_2008}, we assume that the discrete asymptotic nets here are non-degenerate, i.e.\ $\tilde{F}_i$, $\tilde{F}_j$, $\tilde{F}_k$, $\tilde{F}_l$ are non-planar.

For a discrete asymptotic net $\tilde{F}$, the Gauss map $N$ is defined as the unit normal to the tangent plane $\mathcal{P}_{m,n}$.
Similar to discrete curvature lines, discrete asymptotic lines can be defined as follows:
\begin{defi}
	Let $\tilde{F} : \mathbb{Z}^2 \to \mathbb{R}^3$ be a discrete asymptotic net. A discrete space curve $\tilde{F}_{m, n_0}$ (resp. $\tilde{F}_{m_0, n}$) depending on $m$ (resp. $n$) for each $n_0 \in \mathbb{Z}$ (resp. $m_0 \in \mathbb{Z}$) is called a \emph{discrete asymptotic line}.
\end{defi}

Recently, a representation of discrete asymptotic minimal net, where the minimality comes via the edge-constraint condition, was given in \cite[Definition 3.1, Theorem 3.14, Lemma 3.17]{hoffmann_discrete_2017}:
\begin{fact}\label{fact:Weierstrass2}
	For a discrete holomorphic function $g$ with cross ratio factorizing functions $a_{ij}$ and $a_{il}$, a discrete asymptotic net $\tilde{F}$ defined via
	\[\left\{
		\begin{aligned}
			\dif \tilde{F}_{ij} &= a_{ij}\Re\left((1 - g_i g_j, \sqrt{-1}(1 + g_i g_j), g_i + g_j) \frac{\sqrt{-1}}{\dif g_{ij}}\right) \\
			\dif \tilde{F}_{il} &= a_{il}\Re\left((1 - g_i g_l, \sqrt{-1}(1 + g_i g_l), g_i + g_l) \frac{\sqrt{-1}}{\dif g_{il}}\right)
		\end{aligned}
	\right.\]
	becomes a discrete asymptotic minimal net, in the sense of the discrete minimal edge-constraint nets.
\end{fact}


\begin{rema}\label{rema:conjugateRemark}
	It was further shown in \cite[Lemma 3.17]{hoffmann_discrete_2017} that $\tilde{F}$ defined from a discrete holomorphic function $g$ via Fact \ref{fact:Weierstrass2} shares the same unit normal as the discrete isothermic minimal net $F$ defined from the same $g$ via Fact \ref{fact:Weierstrass}. In such case, $\tilde{F}$ is called the conjugate discrete minimal net of $F$.
\end{rema}

\subsection{Reflection properties of discrete minimal nets}
To consider planar discrete space curves, it will be advantageous to use the following notation to denote three consecutive edges:
	\[
		\dif F := F_{m+1, n} - F_{m, n}, \quad \dif F_1 := F_{m+2, n} - F_{m+1, n}, \quad \dif F_{\bar{1}} := F_{m, n} - F_{m-1, n}.
	\]
We first focus on circular nets: let $F$ be a circular net.
Then we have the following lemma, characterizing planar discrete curvature lines in terms of the Gauss map.

\begin{lemm}
	A discrete curvature line on a circular net $F$ is planar if and only if the image of the Gauss map $N$ along the curvature line is contained in a circle.
\end{lemm}
\begin{proof}
	Without loss of generality, the planarity of a discrete curvature line is equivalent to the condition
		\[
			\det(\dif F_{\bar{1}}, \dif F, \dif F_1) = 0
		\]
	on any three consecutive edges.
	However, since $F$ and $N$ are parallel meshes, the above condition is equivalent to
		\[
			\det(\dif N_{\bar{1}}, \dif N, \dif N_1) = 0.
		\]
	Therefore, a discrete curvature line is planar if and only if the image of the Gauss map along the curvature line is planar, i.e.\ contained in a circle.
\end{proof}
Hence, by further requiring that the normal line congruence, i.e.\ the linear span of unit normals placed on the vertices, along the planar curvature line is also included in the same plane, we obtain the following corollary, also mentioned briefly in \cite{bucking_approximation_2007}.

\begin{coro}\label{coro:greatCircle}
	The normal line congruence along a planar discrete curvature line is contained in the same plane if and only if the image of the Gauss map along the curvature line is contained in a great circle.
\end{coro}

Switching our focus to discrete asymptotic nets, now let $\tilde{F}$ be a discrete asymptotic net.
Then we can prove the following lemma characterizing a discrete asymptotic line that is a straight line (see also \cite{bucking_approximation_2007}).
\begin{lemm}\label{lemm:anetGreat}
	A discrete asymptotic line on a discrete asymptotic net $\tilde{F}$ is a straight line if and only if the image of the Gauss map $N$ along the discrete asymptotic line is contained in a great circle.
\end{lemm}

\begin{proof}
	To show one direction, suppose that a discrete asymptotic line $\tilde{F}_{m,n_0}$ is a straight line.
	Then the tangent planes $\mathcal{P}_{m, n_0}$ at each vertex along $\tilde{F}_{m,n_0}$ must include this straight line.
	Therefore, $N_{m, n_0}$ must be contained in the plane perpendicular to the straight line, i.e.\ the image of the Gauss map along the discrete asymptotic line is contained in a great circle.
	
	To show the other direction, now suppose that $N_{m, n_0}$ is contained in a great circle, and let $\mathcal{Q}$ denote the plane containing the great circle with a normal vector $\vec{v}$.
	Then all the tangent planes $\mathcal{P}_{m, n_0}$ must be perpendicular to $\mathcal{Q}$.
	Hence, from the non-degeneracy condition, any two consecutive tangent planes $\mathcal{P}$ and $\mathcal{P}_1$ must intersect along a line parallel to the normal vector $\vec{v}$.
	However, $\mathcal{P}$ and $\mathcal{P}_1$ intersect along the edge $\dif \tilde{F}$, i.e.\ $\dif \tilde{F} \parallel \vec{v}$, and it follows that $\tilde{F}_{m,n_0}$ must be a straight line in the direction of $\vec{v}$. (See Figure \ref{fig:lemmaNet}.)
\end{proof}

\begin{figure}[t]
	\centering
	\includegraphics[scale=0.7]{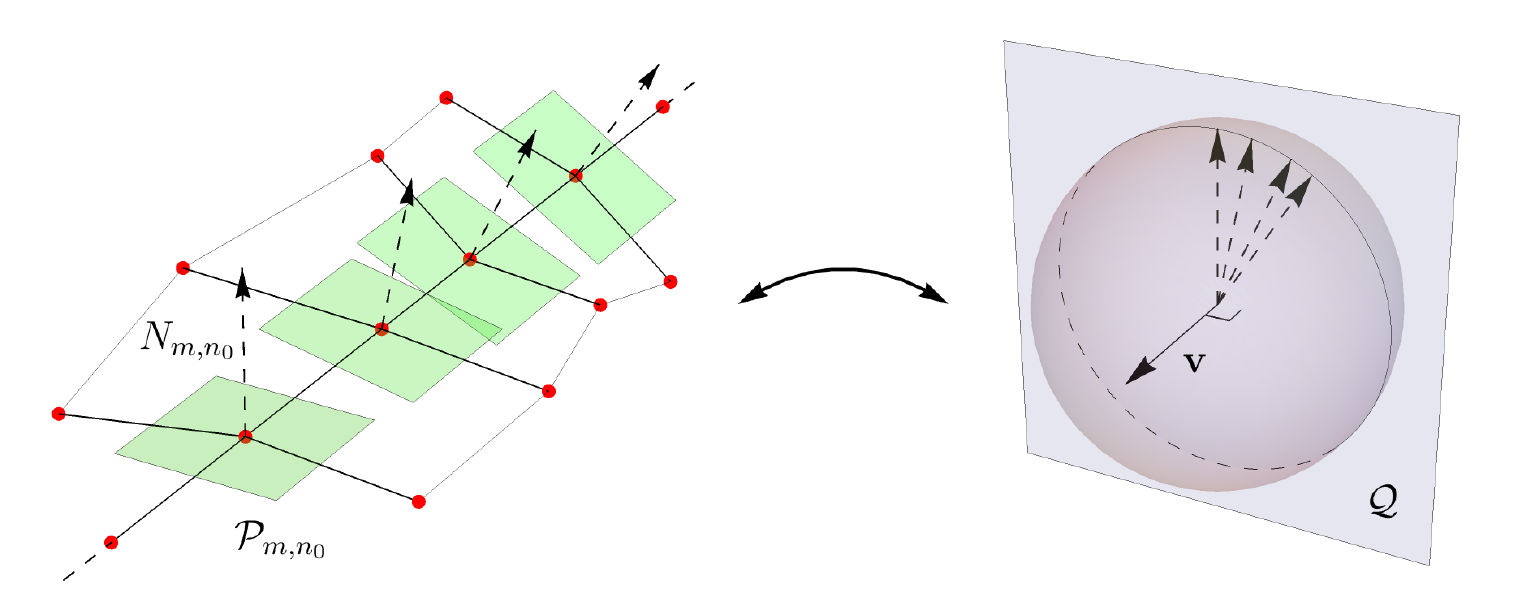}
	\caption{A discrete asymptotic net containing a straight line and its Gauss map.}
	\label{fig:lemmaNet}
\end{figure}
The fact that a discrete isothermic minimal net $F$ and its conjugate discrete asymptotic minimal net $\tilde{F}$ share the same Gauss map $N$, as mentioned in Remark \ref{rema:conjugateRemark}, immediately yields the following corollary.
\begin{coro}\label{coro:linetoplane}
	The normal line congruence along a planar discrete curvature line on a discrete isothermic minimal net $F$ is contained in the same plane if and only if the corresponding discrete asymptotic line on the conjugate discrete asymptotic minimal net $\tilde{F}$ is a straight line.
\end{coro}

Now we prove a reflection principle for discrete asymptotic minimal nets.
Recall that for some $n_0 \in \mathbb{Z}$, $D$ and $\tilde{D}$ were defined as $D := \{(m, n) \in \mathbb{Z}^2 : n \leq n_0\}$ and $\tilde{D} := \{(m, n) \in \mathbb{Z}^2 : n > n_0\}$, respectively.
\begin{theo}\label{theo:asymptoticReflection}
	Let $n_0 \in \mathbb{Z}$, and $\tilde{F} : D \subset \mathbb{Z}^2 \to \mathbb{R}^3$ be a discrete asymptotic minimal net with corresponding Gauss map $N$. Suppose that the discrete asymptotic line $\tilde{F}_{m, n_0}$ is a straight line $\ell$. Extending $\tilde{F}$ to the domain $\mathbb{Z}^2 = D \cup \tilde{D}$ so that the extension is symmetric with respect to the line $\ell$, the extension is a discrete asymptotic minimal net on $\mathbb{Z}^2$.
\end{theo}
\begin{proof}
	Let $F: D \to \mathbb{R}^3$ be the conjugate discrete isothermic minimal net.
	Then by Corollary \ref{coro:linetoplane}, we have that the discrete curvature line $F_{m, n_0}$ and the normal line congruence along the curvature line are contained in the same plane $\mathcal{Q}_1$.
	Therefore, we may invoke Proposition \ref{prop:isothermicReflection} to reflect $F$ across $\mathcal{Q}_1$ so that $F$ and $N$ are now defined on $\mathbb{Z}^2$.
	Now, let $\tilde{F}$ be the conjugate discrete asymptotic minimal net of the extended discrete isothermic minimal net $F$, where $\left.\tilde{F}\right|_D$ agrees with the original $\tilde{F}$.
	We now show that $\tilde{F}$ is symmetric with respect to $\ell$.

	Let $\mathcal{Q}_2$ be the plane such that $N_{m,n_0} \in \mathcal{Q}_2$ for any $m \in \mathbb{Z}$; it follows that $\ell$ is perpendicular to $\mathcal{Q}_2$.
	By construction, $N$ is symmetric with respect to the plane $\mathcal{Q}_2$.
	
	Now, let $T \in \SO$ be a rotation around $\ell$ by 180 degrees, and consider $\hat{F} := T\tilde{F}$.
	By the definition of Gauss maps of discrete asymptotic nets, it must follow that one choice of the Gauss map $\hat{N}$ of $\hat{F}$ be $\hat{N} = -TN$.
	The fact that $\ell$ is perpendicular to $\mathcal{Q}_2$ implies that $\hat{N}_{m,n}$ is symmetric to $N_{m,n}$ with respect to the plane $\mathcal{Q}_2$.
	However, because $N$ is symmetric with respect to $\mathcal{Q}_2$, it follows that $N_{m, n_0 + k} = \hat{N}_{m, n_0 - k}$.
	Since, $\tilde{F}$ and $\hat{F}$ share the same initial condition along $\ell$, we have $\tilde{F}_{m,n_0 + k} = \hat{F}_{m, n_0 - k}$ by Fact \ref{fact:Weierstrass2}. 
\end{proof}

%
%

\section{Examples of discrete minimal nets with symmetry}\label{sect:examples}

Let $F : \mathbb{Z}^2 \to \mathbb{R}^3$ be a discrete isothermic minimal surface with Gauss map $N$, and choose a point $(m_0, n_0) \in \mathbb{Z}^2$. 
Suppose that the discrete curves $F_{m, n_0}$ and $F_{m_0, n}$, and also the normal line congruences along these curves, are contained in the planes $\mathcal{P}_1$ and $\mathcal{P}_2$, respectively.
Since we have that $F$ and $N$ are edge-parallel, $N_{m, n_0}$ and $N_{m_0, n}$ must also be contained in planes $\mathcal{Q}_1$ and $\mathcal{Q}_2$ containing the origin and parallel to $\mathcal{P}_1$ and $\mathcal{P}_2$, respectively (see also Corollary \ref{coro:greatCircle}).
Denoting the quadrilateral $(m_0,n_0), (m_0+1, n_0), (m_0 + 1, n_0 + 1), (m_0, n_0 + 1)$ by $(ijk\ell)$, we have the following lemma.

\begin{lemm}
	The angle between the planes $\mathcal{P}_1$ and $\mathcal{P}_2$ measured on the side containing the quadrilateral $F_{ijk\ell}$, and the angle between the planes $\mathcal{Q}_1$ and $\mathcal{Q}_2$ measured on the side containing the quadrilateral $N_{ijk\ell}$ are supplementary angles.
\end{lemm}

\begin{proof}
	Since $F$ and $N$ are parallel meshes, the angle between $\mathcal{P}_1$ and $\mathcal{P}_2$ equals that between $\mathcal{Q}_1$ and $\mathcal{Q}_2$.
	However, the Christoffel duality, or the Weierstrass representation,  tells us that the orientations of $F$ and $N$ are opposite, giving us the desired conclusion. 
\end{proof}

\begin{rema}
	Since stereographic projection is a Möbius transformation, it preserves angles.
	Therefore, to determine the angle between $\mathcal{Q}_1$ and $\mathcal{Q}_2$, one only needs to look at the angle between the circles containing $g_{m,n_0}$ and $g_{m_0, n}$.
\end{rema}

Before looking at the examples, we comment on how to change the Weierstrass data of a given smooth minimal surface so that it is parametrized with isothermic coordinates (see, for example, \cite[Section 2.3]{bobenko_painleve_2000}).
Let a (smooth) minimal surface $X : \Sigma \subset \mathbb{R}^2 \cong \mathbb{C} \to \mathbb{R}^3$ be represented by
	\[
		X(z) = \Re \int (1 - g(z)^2, \sqrt{-1}(1 + g(z)^2), 2g(z)) f(z) \dif z
	\]
over a simply-connected domain $\Sigma$ on which $g$ is meromorphic, while $f$ and $fg^2$ are holomorphic.
Then the coordinate $w$ satisfying
	\begin{equation}\label{eqn:coordinateChange}
		(w_z)^2 = f g_z \quad\text{\resp{$(w_z)^2 = -\sqrt{-1} f g_z$}},
	\end{equation}
for $w_z = \frac{\partial w}{\partial z}$, becomes an isothermic \resp{conformal asymptotic} coordinate of $X$, and $X$ can be represented as
	\begin{gather*}
		X(w) = \Re \int (1 - g(w)^2, \sqrt{-1}(1 + g(w)^2), 2g(w)) \frac{1}{g_w(w)} \dif w\\
		\text{\resp{$X(w) = \Re \int (1 - g(w)^2, \sqrt{-1}(1 + g(w)^2), 2g(w)) \frac{\sqrt{-1}}{g_w(w)} \dif w$}}.
	\end{gather*}
\begin{exam}
	Recall that the well-known Enneper surface and higher order Enneper surfaces can be represented via the Weierstrass data $g(z) = z^k$ and $f(z) = 1$ for $k \in \mathbb{N}$.
	Taking the coordinate change as in \eqref{eqn:coordinateChange} (and applying a suitable homothety on the domain depending on $k$), we obtain new Weierstrass data $g(w) = w^\frac{2k}{k+1}$.
	
	Therefore, from the discrete power function $z^\gamma$ defined in \cite{agafonov_discrete_2000} (see also \cite{ando_explicit_2014, hoffmann_discrete_2012}), let $g$ be the discrete power function with $\gamma = \frac{2k}{k+1}$.
	Then, $g_{m,0} \in \mathbb{R}_{\geq 0}$ while $g_{0,n}$ is on the line $z = r e^{\sqrt{-1} \frac{k \pi }{k+1}}$ for $r \in \mathbb{R}_{\geq 0}$.
	Hence, $F_{m,0}$ and $F_{0,n}$ are on planes meeting at an angle $\frac{\pi}{k+1}$.
	Reflecting the surface iteratively with respect to these planes give us the discrete isothermic analogue of higher order Enneper surfaces, and by considering its conjugate via Fact \ref{fact:Weierstrass2}, we obtain a discrete asymptotic net with line symmetries (see Figure \ref{fig:discEnneper}).

\begin{figure}[t]
	\centering
	\includegraphics[scale=0.6]{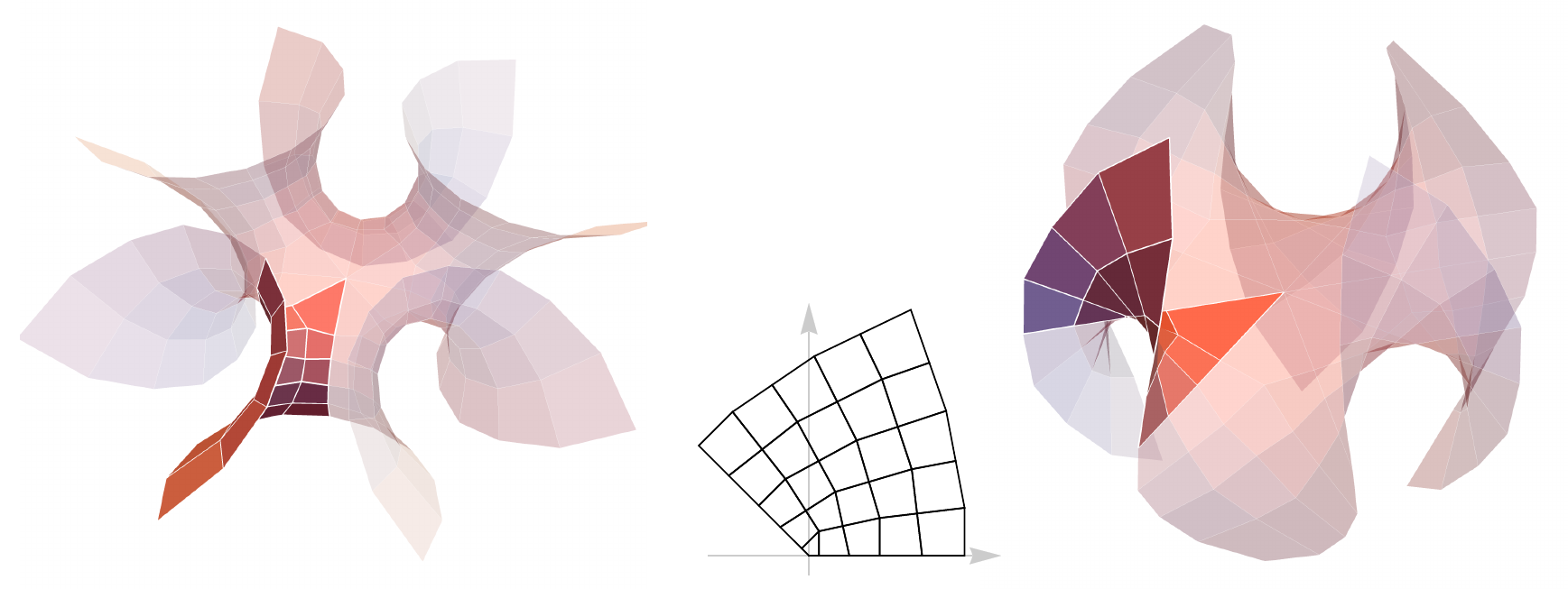}
	\caption{A discrete higher order Enneper surface from a discrete power function and its conjugate. The left-hand side is a discrete isothermic minimal net having planar symmetry; the right-hand side is a discrete asymptotic minimal net having line symmetry. This figure was drawn with $k = 3$. (See also \cite{rossman_discrete_2018}.)}
	\label{fig:discEnneper}
\end{figure}

\end{exam}

\begin{exam}
	Planar Enneper surfaces (see, for example \cite{karcher_construction_1989}) are examples of minimal surfaces with planar ends.
	In particular, the planar Enneper surface with $2$-fold symmetry is given by the Weierstrass data $g(z) = z^3$ and $f(z) = \frac{1}{g_z(z)}$; hence, $z$ is an isothermic coordinate.
	
	The discrete power function $z^3$ following \cite{agafonov_discrete_2000, ando_explicit_2014, hoffmann_discrete_2012} becomes immersed on the domain $D := \{ (m,n) \in \mathbb{Z}^2 : m \geq 0, n \geq 0\} \setminus \{(0,0)\}$, and $g_{m,0} \in \mathbb{R}$ while $g_{0,n}$ is on the line $z = - r \sqrt{-1}$ for $r \in \mathbb{R}_{>0}$.
	Therefore, $F_{m,0}$ and $F_{0,n}$ are on planes meeting at an angle $\frac{\pi}{2}$, and the resulting surface has $2$-fold symmetry, and by considering its conjugate via Fact \ref{fact:Weierstrass2}, we obtain an example of a discrete asymptotic net with line symmetries (see Figure \ref{fig:discz3Rich}).

\begin{figure}[t]
	\centering
	\includegraphics[scale=0.6]{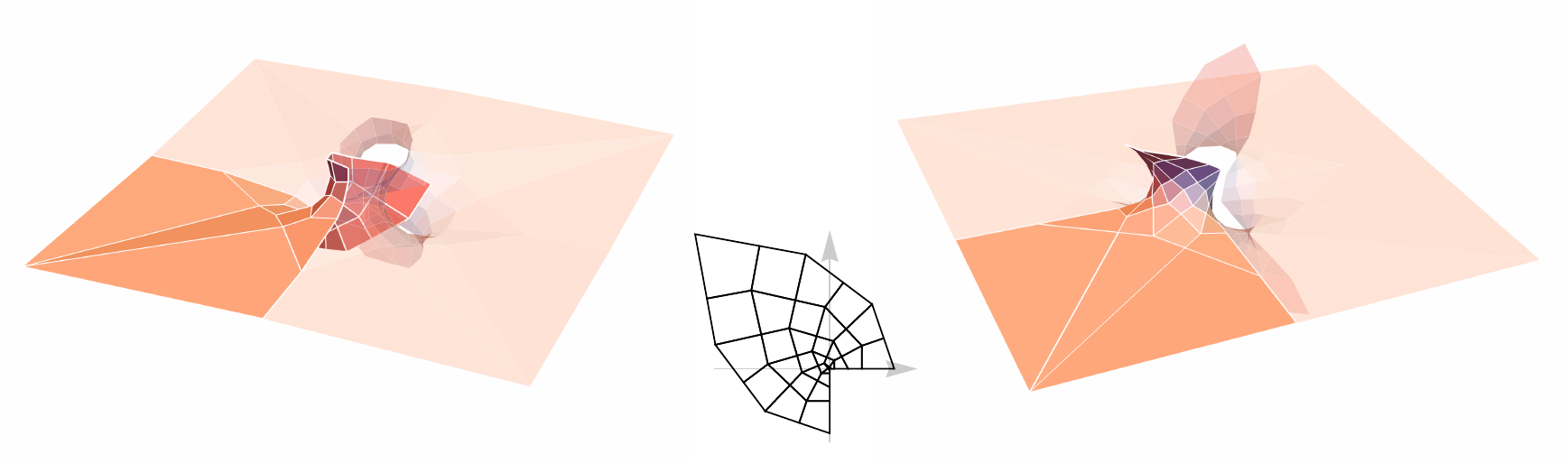}
	\caption{Discrete planar Enneper surface with $2$-fold symmetry from discrete power function $z^3$ and its conjugate. The left-hand side is a discrete isothermic minimal net having planar symmetry; the right-hand side is a discrete asymptotic minimal net having line symmetry.}
	\label{fig:discz3Rich}
\end{figure}

\end{exam}

\begin{exam}
	The minimal $k$-noids (for $k \in \mathbb{N}$, $k \geq 3$) of Jorge-Meeks in \cite{jorge_topology_1983} are minimal surfaces that are topologically equivalent to the sphere minus $k$ disks with $k$ catenoidal ends, given by the Weierstrass data $g(z) = z^{k-1}$ and $f(z) = \frac{1}{(z^k - 1)^2}$.
	Changing coordinates as in \eqref{eqn:coordinateChange} (and applying a suitable homothety on the domain depending on $k$), we obtain new Weierstrass data $g(w) = (\tanh w)^{\frac{2k - 2}{k}}$ with isothermic coordinate $w$.
	Under such settings, a fundamental piece of the minimal $k$-noid can be drawn over the region $w \in [0, \infty] \times \left[0, \tfrac{\pi}{4}\right] \subset \mathbb{R}^2 \cong \mathbb{C}$
	over which $g(w)$ has values
		\[
			g(w) \in D_k := \left\{z = r e^{\sqrt{-1} \theta} : 0 \leq r \leq 1, 0 \leq \theta \leq \tfrac{(k-1)\pi}{k} \right\} \setminus \{ 1 \}.
		\]
	In fact, as also demonstrated in Figure \ref{fig:gfornnoid},
		\[
			g(w) = 
				\begin{cases}
					r(w), & \text{if $w \in [0, \infty] \times \{0\}$ }\\
					r(w)e^{\sqrt{-1} \frac{(k-1) \pi}{k}}, & \text{if $w \in \{0\} \times \left[0, \frac{\pi}{4}\right]$}\\
					e^{\sqrt{-1} \theta(w)}, & \text{if $w \in [0, \infty] \times\{\frac{\pi}{4}\}$}.
				\end{cases}
		\]
	\begin{figure}
		\centering
	\begin{minipage}{0.32\textwidth}
		\includegraphics[width=0.9\textwidth]{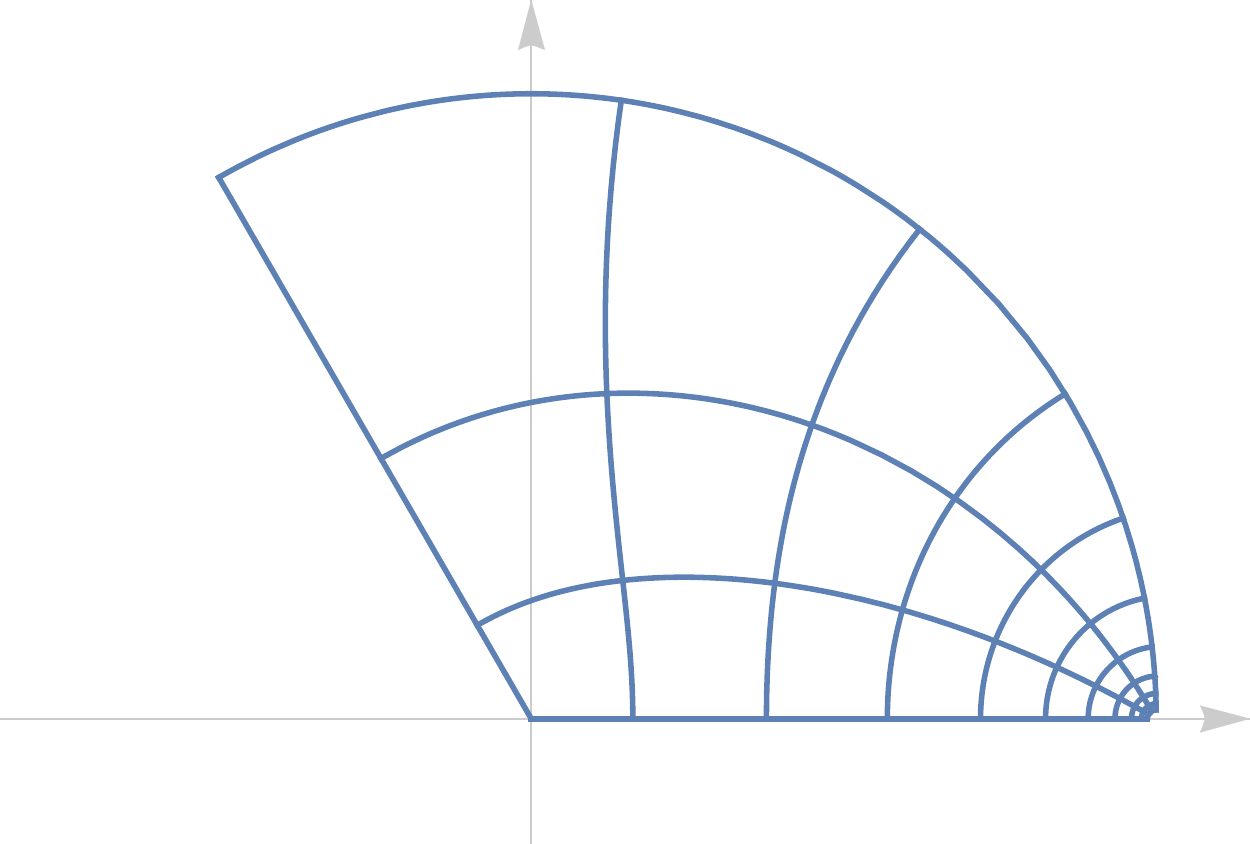}
	\end{minipage}
	\begin{minipage}{0.32\textwidth}
		\includegraphics[width=0.9\textwidth]{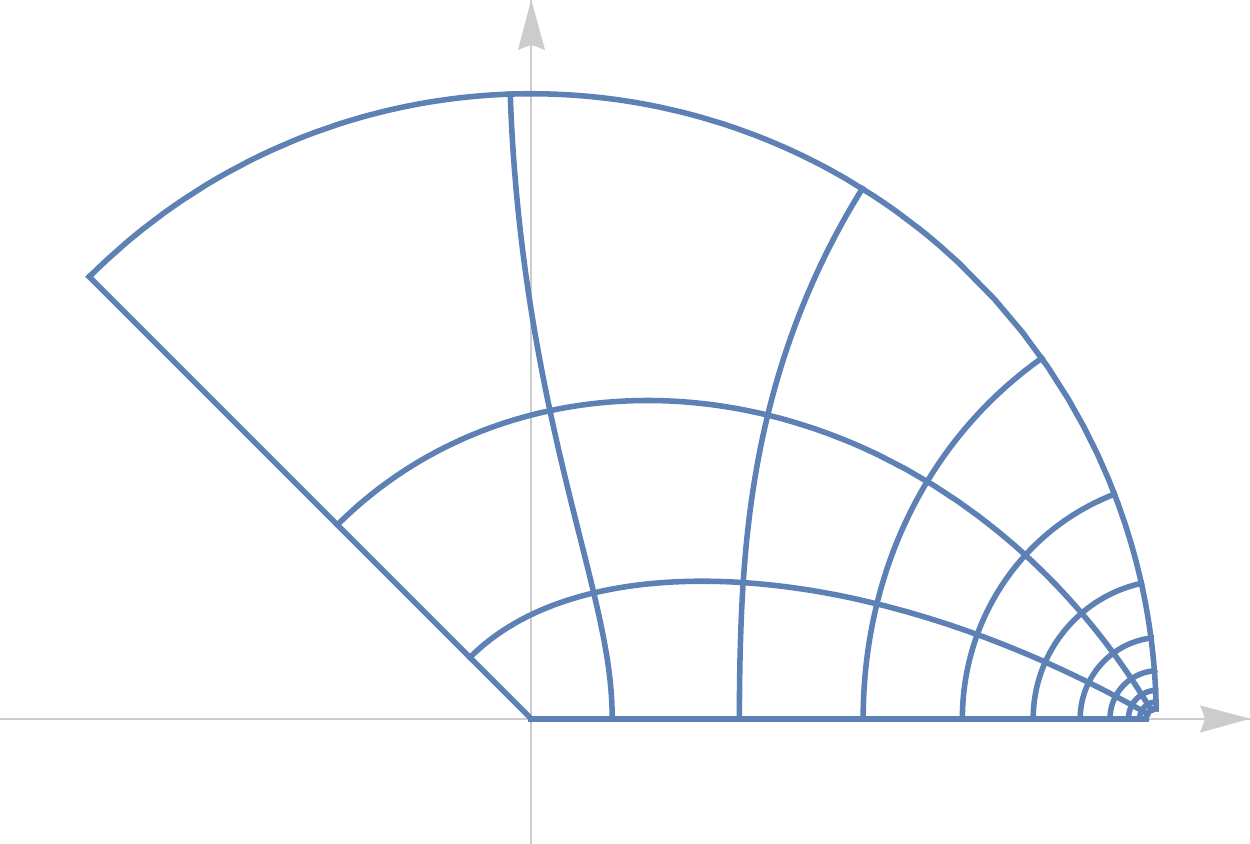}
	\end{minipage}
	\begin{minipage}{0.32\textwidth}
		\includegraphics[width=0.9\textwidth]{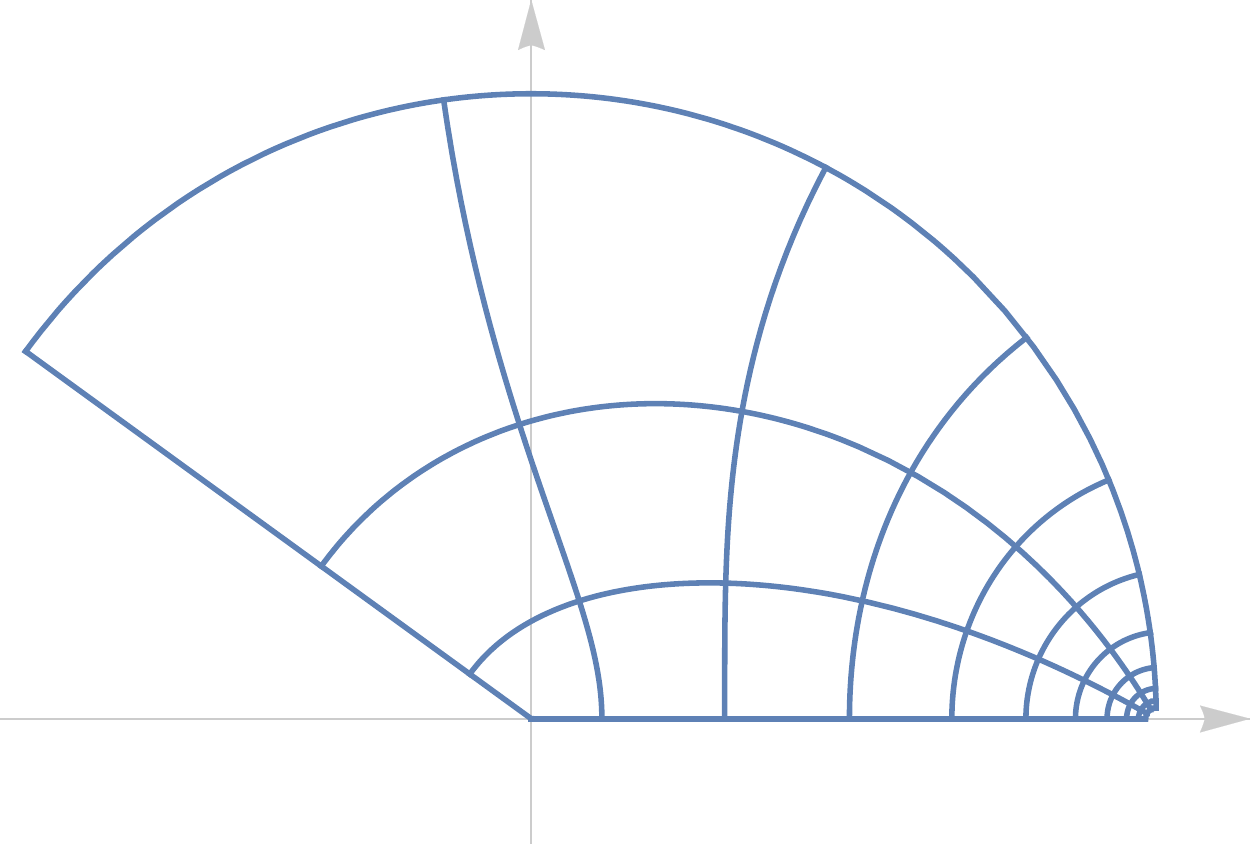}
	\end{minipage}
	\caption{Images of smooth $g(w)$ giving fundamental pieces of the minimal $k$-noids, drawn for $k = 3,4,5$.}
	\label{fig:gfornnoid}
\end{figure}

	To discretize $g$ (numerically) over the domain $\{ (m,n) \in \mathbb{Z}^2 : m \geq 0, 0 \leq n \leq n_{\text{max}}\}$, we require that
	\begin{itemize}
		\item $g_{0,0} = 0$ and $g_{0,n_{\text{max}}} = e^{\sqrt{-1}\frac{(k-1)\pi}{k}}$,
		\item $g_{m,0} \in [0,1)$ is a strictly increasing sequence,
		\item $g_{0,n} = r_n e^{\sqrt{-1}\frac{(k-1)\pi}{k}}$ where $r_n \in [0,1]$ is a strictly increasing finite sequence,
		\item $g_{m,n_{\text{max}}} = e^{\sqrt{-1}\theta_m}$ where $\theta_m \in \left(0,\tfrac{(k-1)\pi}{k}\right]$ is a strictly decreasing sequence,
		\item the cross ratio of $g$ over any elementary quadrilateral is equal to $-1$, and
		\item $g_{m,n} \in D_k$ for all $(m,n)$ in the domain.
	\end{itemize}
	By the definition of $g$, we know that
	\begin{itemize}
		\item the planes containing $F_{m,0}$ and $F_{0,n}$ meet at an angle $\frac{\pi}{k}$, and
		\item the planes containing $F_{n,0}$ and $F_{m,n_{\text{max}}}$ meet at an angle $\frac{\pi}{2}$,
	\end{itemize}
	giving us a discrete analogue of minimal $k$-noids of Jorge-Meeks (see Figures \ref{fig:disc3noid} and \ref{fig:disc45noid}).
	
\begin{figure}[t]
	\centering
	\includegraphics[scale=0.45]{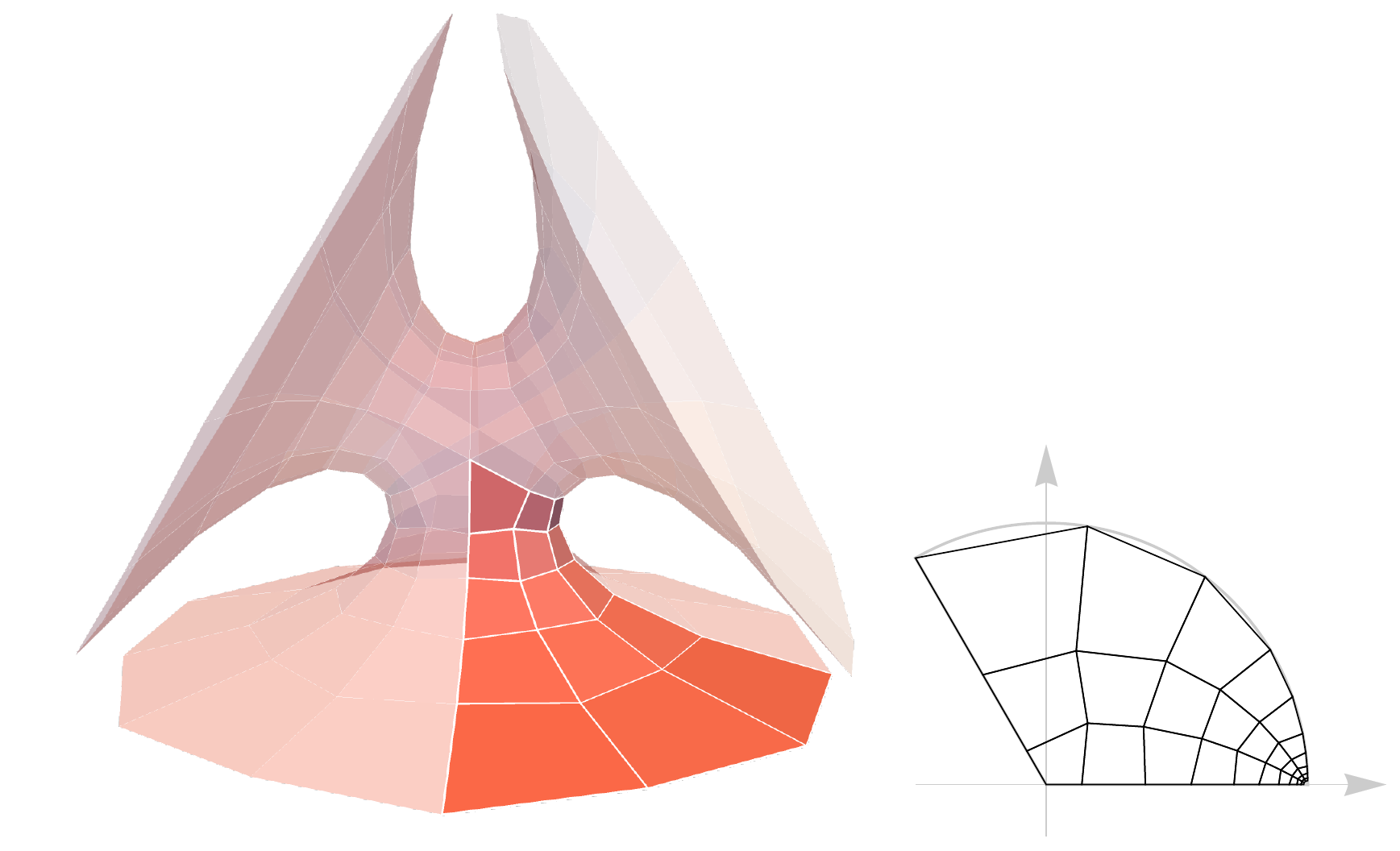}
	\caption{Numerical solution of discrete trinoid ($k = 3$) given with its discrete holomorphic function satisfying the boundary conditions (with $n_{\text{max}} = 3$).}
	\label{fig:disc3noid}
\end{figure}	
	
\begin{figure}[t]
	\centering
	\begin{minipage}{0.48\textwidth}
		\includegraphics[width=\textwidth]{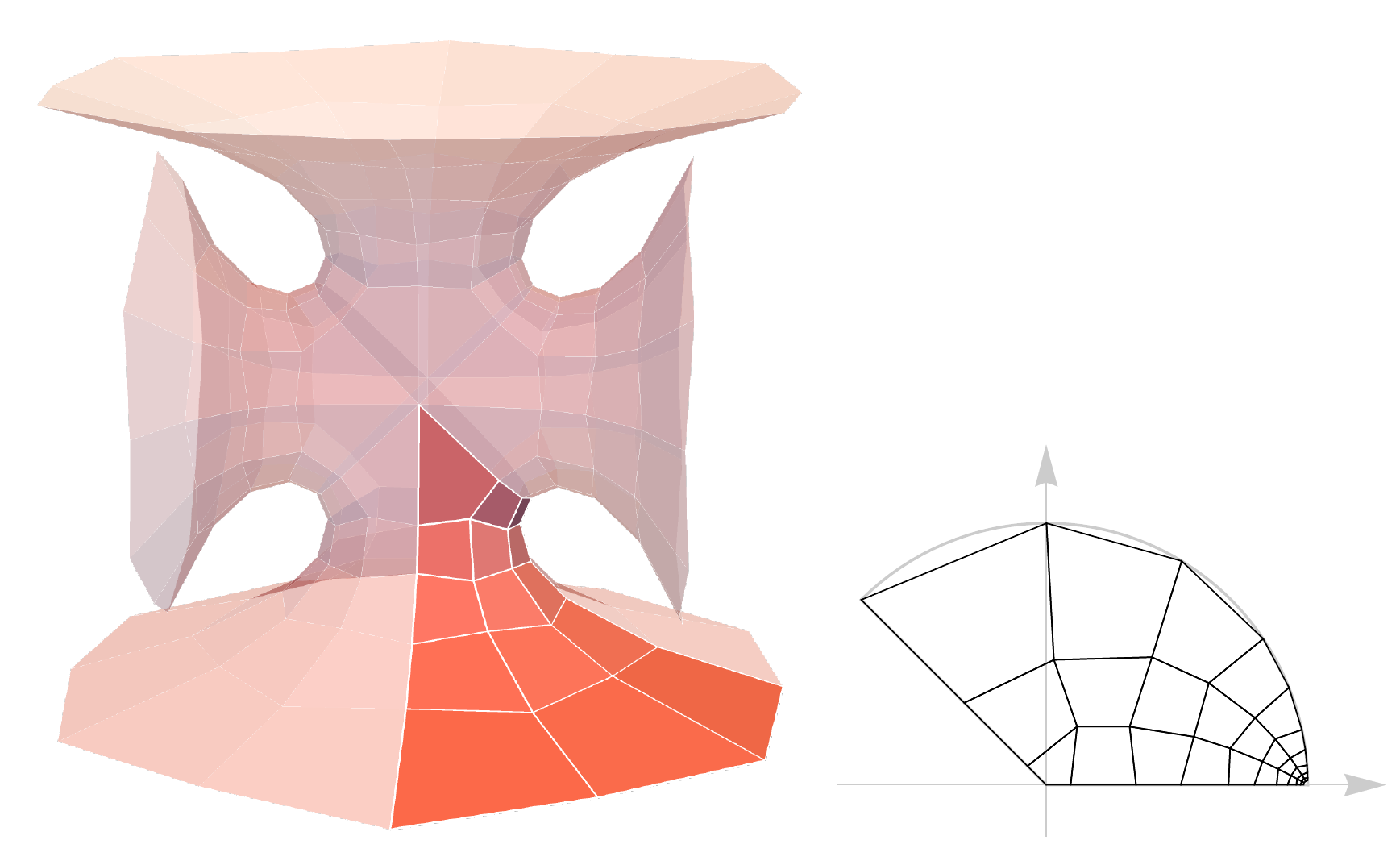}
	\end{minipage}
	\begin{minipage}{0.48\textwidth}
		\includegraphics[width=\textwidth]{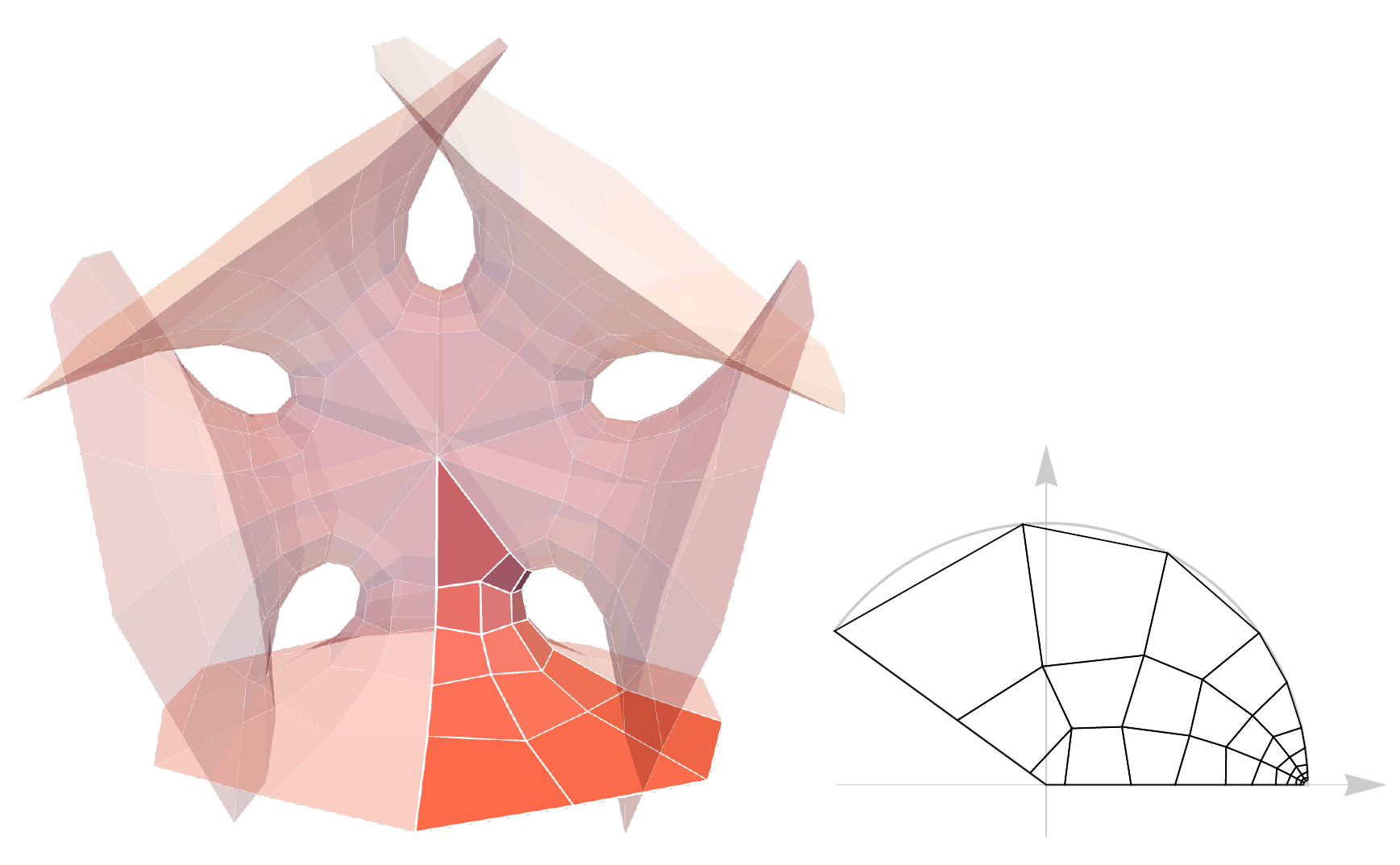}
	\end{minipage}
	\caption{Numerical solutions of discrete $4$-noid and $5$-noid given with their discrete holomorphic functions satisfying the boundary conditions (with $n_{\text{max}} = 3$).}
	\label{fig:disc45noid}
\end{figure}		
\end{exam}

\begin{exam}
	By expanding on the idea of using the symmetry of $k$-noids as boundary conditions for the holomorphic data, we can create other discrete minimal nets with symmetries.
	In this example, we create discrete minimal nets with symmetry groups of the Platonic solids \cite{xu_symmetric_1995}.
	As in the $k$-noids examples, we can ascertain the boundary conditions from the symmetries of the discrete minimal net by calculating the angles at which the great circles meet (see, for example, \cite{berglund_minimal_1995}).
	Then, by finding discrete holomorphic functions satisfying the given boundary conditions, we can obtain discrete minimal nets with symmetry groups of the Platonic solids.
	Here, we show two numerical examples of discrete minimal nets with such symmetries in Figure \ref{fig:solids}.
	\begin{figure}[t]
		\centering
		\begin{minipage}{0.4\textwidth}
			\includegraphics[width=\textwidth]{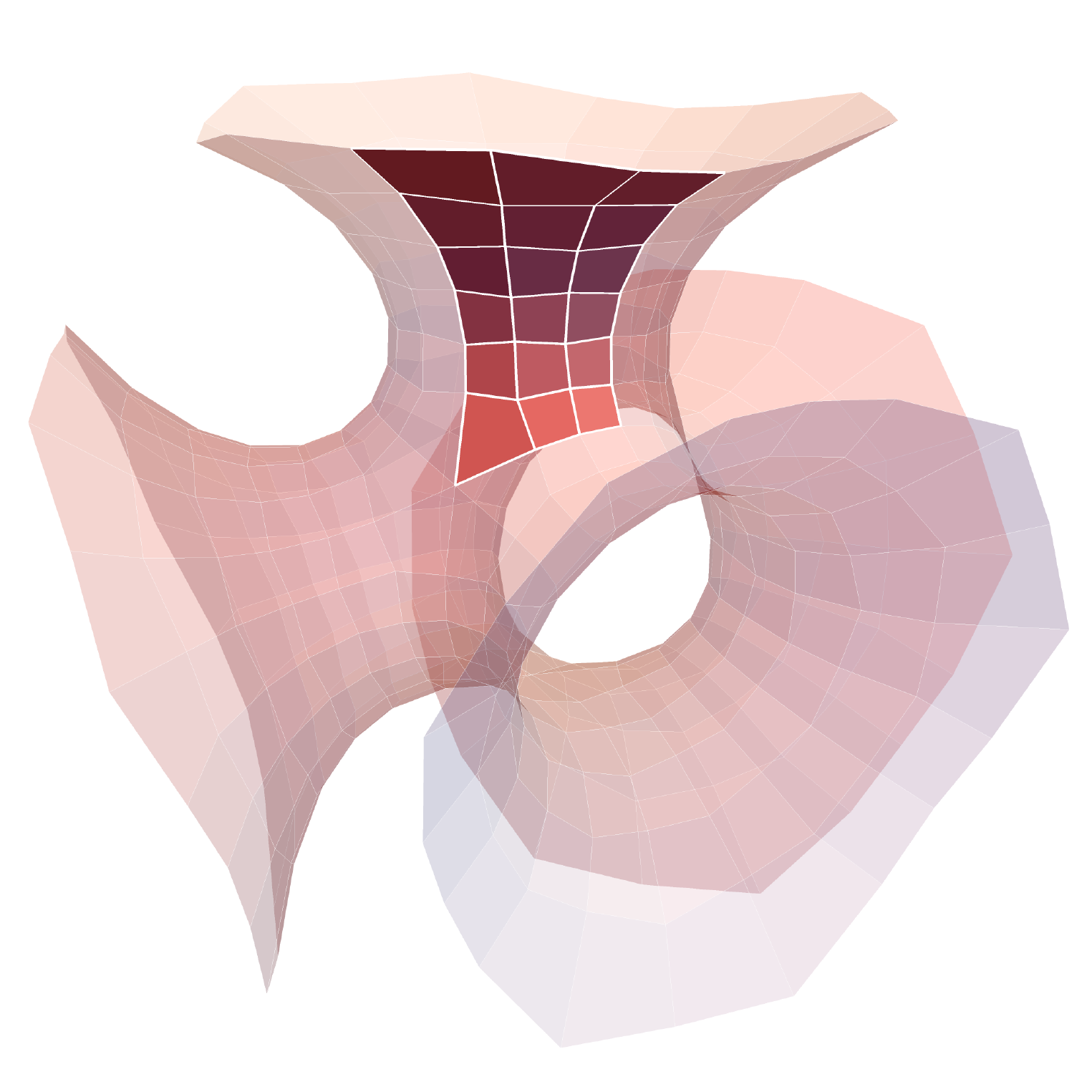}
		\end{minipage}
		\begin{minipage}{0.08\textwidth}
			\quad
		\end{minipage}
		\begin{minipage}{0.4\textwidth}
			\includegraphics[width=\textwidth]{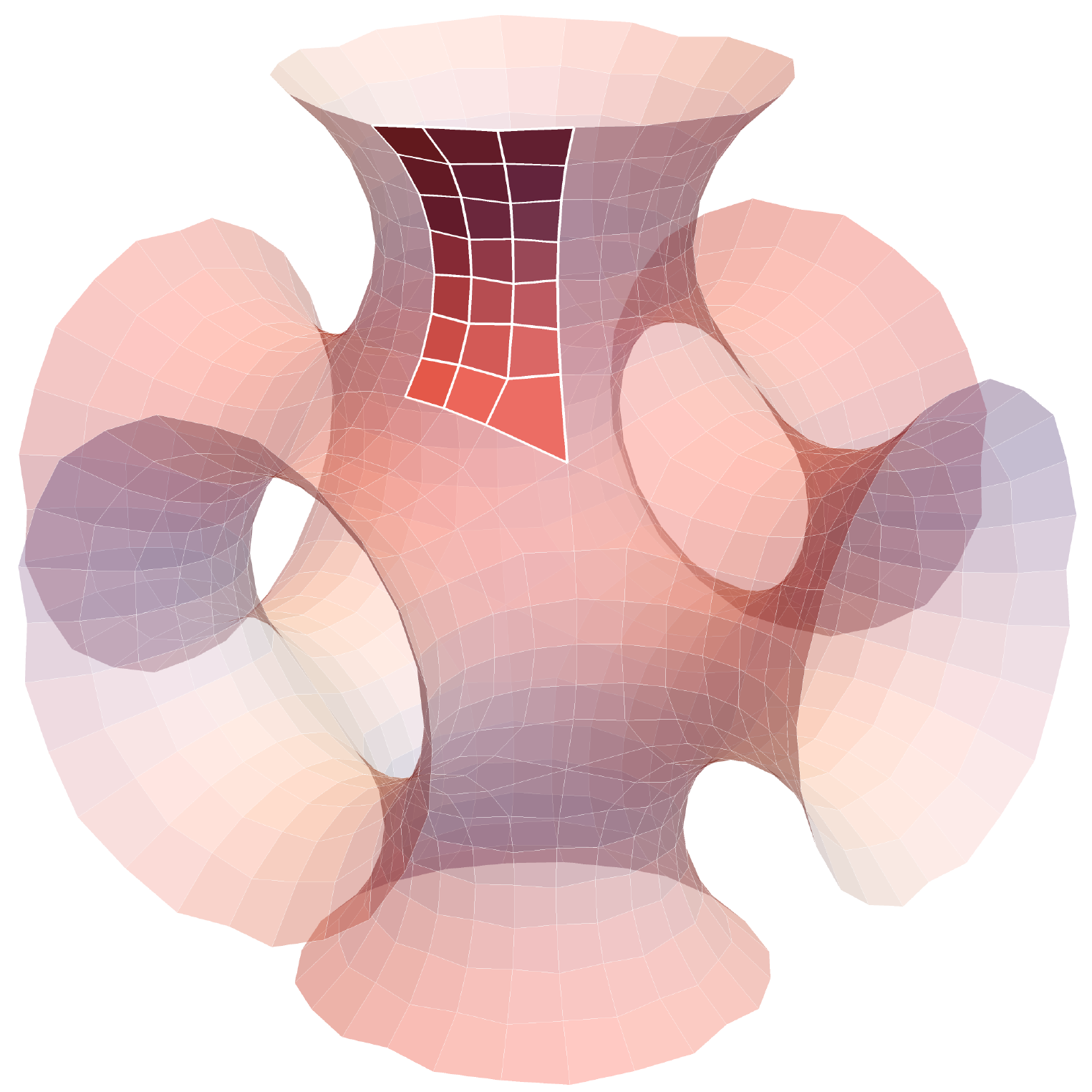}
		\end{minipage}
		\caption{Numerical solutions of discrete minimal nets with tetrahedral symmetry on the left, and octahedral symmetry on the right.}
		\label{fig:solids}
	\end{figure}	
\end{exam}

\begin{rema}
	One may notice that while most of the vertices on the examples have degree 4, i.e.\ $4$ edges meet at the vertex, there are vertices with degree higher than $4$.
	While this may indicate the existence of a branch point on the Gauss map, we have avoided this issue by assigning these vertices to be one of the ``corner'' points of the fundamental piece, and treating the Gauss map as coming from a holomorphic function on a simply-connected domain in the complex plane.
	In fact, on these vertices, the definition of discrete minimality as in \cite[Definition 7]{bobenko_discrete_1996-1} might not be directly applicable; however, the definition via Steiner's formula (as in Definition \ref{def:curvatures} and Definition \ref{def:minimality}) allows us to consider mean curvatures on the faces around such points, and determine minimality at these points as well.
\end{rema}

\vspace{15pt}
\textbf{Acknowledgements.}
The authors would like to express their gratitude to Professor Masashi Yasumoto for fruitful discussions, and the referee for valuable comments. The first author was partially supported by JSPS/FWF Bilateral Joint Project I3809-N32 ``Geometric shape generation'' and Grant-in-Aid for JSPS Fellows No.\ 19J10679; the second author was partially supported by two JSPS grants, Grant-in-Aid for Scientific Research (C) 15K04845 and (S) 17H06127 (P.I.: M.-H.\ Saito); the third author was partially supported by NRF 2017 R1E1A1A 03070929.
\nocite{*}

\begin{bibdiv}
\begin{biblist}

\bib{agafonov_discrete_2000}{article}{
      author={Agafonov, Sergey~I.},
      author={Bobenko, Alexander~I.},
       title={Discrete {$Z^\gamma$} and {Painlevé} equations},
        date={2000},
     journal={Internat. Math. Res. Notices},
      volume={2000},
      number={4},
       pages={165–193},
       doi={10.1155/S1073792800000118},
}

\bib{ando_explicit_2014}{article}{
      author={Ando, Hisashi},
      author={Hay, Mike},
      author={Kajiwara, Kenji},
      author={Masuda, Tetsu},
       title={An explicit formula for the discrete power function associated
  with circle patterns of {{Schramm}} type},
        date={2014},
     journal={Funkcial. Ekvac.},
      volume={57},
      number={1},
       pages={1\ndash 41},
       doi={10.1619/fesi.57.1},
}

\bib{berglund_minimal_1995}{article}{
      author={Berglund, Jorgen},
      author={Rossman, Wayne},
       title={Minimal surfaces with catenoid ends},
        date={1995},
     journal={Pacific J. Math.},
      volume={171},
      number={2},
       pages={353\ndash 371},
       doi={10.2140/pjm.1995.171.353},
}

\bib{bobenko_discrete_2017}{incollection}{
      author={Bobenko, Alexander~I.},
      author={B\"ucking, Ulrike},
      author={Sechelmann, Stefan},
       title={Discrete minimal surfaces of {{Koebe}} type},
        date={2017},
   booktitle={Modern approaches to discrete curvature},
      editor={Najman, Laurent},
      editor={Romon, Pascal},
      series={Lecture Notes in Math.},
      volume={2184},
   publisher={{Springer}},
     address={Cham},
       pages={259\ndash 291},
       doi={10.1007/978-3-319-58002-9_8},
}

\bib{bobenko_painleve_2000}{book}{
      author={Bobenko, Alexander~I.},
      author={Eitner, Ulrich},
       title={Painlev\'e equations in the differential geometry of surfaces},
      series={Lecture Notes in Math.},
   publisher={{Springer-Verlag}},
     address={Berlin},
        date={2000},
      volume={1753},
        doi={10.1007/b76883},
}

\bib{bobenko_discrete_1996-1}{article}{
      author={Bobenko, Alexander~I.},
      author={Pinkall, Ulrich},
       title={Discrete isothermic surfaces},
        date={1996},
     journal={J. Reine Angew. Math.},
      volume={475},
       pages={187\ndash 208},
       doi={10.1515/crll.1996.475.187},
}


\bib{bobenko_curvature_2010}{article}{
      author={Bobenko, Alexander~I.},
      author={Pottmann, Helmut},
      author={Wallner, Johannes},
       title={A curvature theory for discrete surfaces based on mesh
  parallelity},
        date={2010},
        ISSN={0025-5831},
     journal={Math. Ann.},
      volume={348},
      number={1},
       pages={1\ndash 24},
       doi={10.1007/s00208-009-0467-9},
}

\bib{bobenko_discrete_2008}{book}{
      author={Bobenko, Alexander~I.},
      author={Suris, Yuri~B.},
       title={Discrete differential geometry},
      series={Graduate Studies in Mathematics},
   publisher={{American Mathematical Society}},
     address={Providence, RI},
        date={2008},
      number={98},
        ISBN={978-0-8218-4700-8},
}

\bib{bucking_approximation_2007}{thesis}{
      author={B\"ucking, Ulrike},
       title={Approximation of conformal mappings by circle patterns and
  discrete minimal surfaces},
        type={Ph.{{D}}. {{Thesis}}},
        date={2007},
}

\bib{bucking_minimal_2008}{incollection}{
      author={B\"ucking, Ulrike},
       title={Minimal surfaces from circle patterns: Boundary value problems,
  examples},
        date={2008},
   booktitle={Discrete differential geometry},
      editor={Bobenko, Alexander~I.},
      editor={Schr\"oder, Peter},
      editor={Sullivan, John~M.},
      editor={Ziegler, G\"unter~M.},
      series={Oberwolfach Semin.},
      volume={38},
   publisher={{Birkh\"auser}},
     address={Basel},
       pages={37\ndash 56},
       doi={10.1007/978-3-7643-8621-4_2},
}

\bib{burstall_discrete_2018}{article}{
      author={Burstall, Francis~E.},
      author={{Hertrich-Jeromin}, Udo},
      author={Rossman, Wayne},
       title={Discrete linear {{Weingarten}} surfaces},
        date={2018},
        ISSN={0027-7630},
     journal={Nagoya Math. J.},
      volume={231},
       pages={55\ndash 88},
       doi={10.1017/nmj.2017.11},
}

\bib{costa_example_1984}{article}{
      author={Costa, Celso~J.},
       title={Example of a complete minimal immersion in $\bf {R}^3$ of genus
  one and three embedded ends},
        date={1984},
        ISSN={0100-3569},
     journal={Bol. Soc. Brasil. Mat.},
      volume={15},
      number={1-2},
       pages={47–54},
       doi={10.1007/BF02584707},
}

\bib{doliwa_transformations_2000}{article}{
      author={Doliwa, Adam},
      author={Santini, Paolo~Maria},
      author={Ma\~nas, Manuel},
       title={Transformations of quadrilateral lattices},
        date={2000},
        ISSN={0022-2488},
     journal={J. Math. Phys.},
      volume={41},
      number={2},
       pages={944\ndash 990},
       doi={10.1063/1.533175},
}

\bib{ejiri_remark_2015}{article}{
      author={Ejiri, Norio},
      author={Fujimori, Shoichi},
      author={Shoda, Toshihiro},
       title={A remark on limits of triply periodic minimal surfaces of genus
  3},
        date={2015},
        ISSN={0166-8641},
     journal={Topology Appl.},
      volume={196},
      number={part B},
       pages={880\ndash 903},
       doi={10.1016/j.topol.2015.05.014},
}

\bib{hoffman_complete_1985}{article}{
      author={Hoffman, David~A.},
      author={Meeks, William~H., III},
       title={A complete embedded minimal surface in $\bf {R}^3$ with genus one
  and three ends},
        date={1985},
        ISSN={0022-040X},
     journal={J. Differential Geom.},
      volume={21},
      number={1},
       pages={109–127},
       doi={10.4310/jdg/1214439467},
}

\bib{hoffman_embedded_1990}{article}{
      author={Hoffman, David~A.},
      author={Meeks, William~H., III},
       title={Embedded minimal surfaces of finite topology},
        date={1990},
        ISSN={0003-486X},
     journal={Ann. of Math. (2)},
      volume={131},
      number={1},
       pages={1\ndash 34},
       doi={10.2307/1971506},
}

\bib{hoffmann_discrete_2000}{thesis}{
      author={Hoffmann, Tim},
       title={Discrete curves and surfaces},
        type={Ph.{{D}}. {{Thesis}}},
        date={2000},
}

\bib{hoffmann_discrete_2012}{article}{
      author={Hoffmann, Tim},
      author={Rossman, Wayne},
      author={Sasaki, Takeshi},
      author={Yoshida, Masaaki},
       title={Discrete flat surfaces and linear {{Weingarten}} surfaces in
  hyperbolic 3-space},
        date={2012},
        ISSN={0002-9947},
     journal={Trans. Amer. Math. Soc.},
      volume={364},
      number={11},
       pages={5605\ndash 5644},
       doi={10.1090/S0002-9947-2012-05698-4},
}

\bib{hoffmann_discrete_2017}{article}{
      author={Hoffmann, Tim},
      author={Sageman-Furnas, Andrew~O.},
      author={Wardetzky, Max},
       title={A discrete parametrized surface theory in $\mathbb{R}^3$},
        date={2017},
        ISSN={1073-7928},
     journal={Int. Math. Res. Not. IMRN},
      volume={2017},
      number={14},
       pages={4217–4258},
       doi={10.1093/imrn/rnw015},
}

\bib{jorge_topology_1983}{article}{
      author={Jorge, Luqu\'esio~P.},
      author={Meeks, William~H., III},
       title={The topology of complete minimal surfaces of finite total
  {{Gaussian}} curvature},
        date={1983},
        ISSN={0040-9383},
     journal={Topology},
      volume={22},
      number={2},
       pages={203\ndash 221},
       doi={10.1016/0040-9383(83)90032-0},
}

\bib{karcher_embedded_1988}{article}{
      author={Karcher, Hermann},
       title={Embedded minimal surfaces derived from {{Scherk}}'s examples},
        date={1988},
        ISSN={0025-2611},
     journal={Manuscripta Math.},
      volume={62},
      number={1},
       pages={83\ndash 114},
       doi={10.1007/BF01258269},
}

\bib{karcher_construction_1989}{article}{
      author={Karcher, Hermann},
       title={Construction of minimal surfaces},
        date={1989},
     journal={Preprint No. 12, SFB 256, Universit\"at Bonn},
}

\bib{karcher_triply_1989}{article}{
      author={Karcher, Hermann},
       title={The triply periodic minimal surfaces of {{Alan Schoen}} and their
  constant mean curvature companions},
        date={1989},
        ISSN={0025-2611},
     journal={Manuscripta Math.},
      volume={64},
      number={3},
       pages={291\ndash 357},
       doi={10.1007/BF01165824},
}

\bib{karcher_construction_1996}{article}{
      author={Karcher, Hermann},
      author={Polthier, Konrad},
       title={Construction of triply periodic minimal surfaces},
        date={1996},
        ISSN={0962-8428},
     journal={Philos. Trans. Roy. Soc. London Ser. A},
      volume={354},
      number={1715},
       pages={2077\ndash 2104},
       doi={10.1098/rsta.1996.0093},
}

\bib{koch_$3$-periodic_1988}{article}{
      author={Koch, Elke},
      author={Fischer, Werner},
       title={On $3$-periodic minimal surfaces with noncubic symmetry},
        date={1988},
        ISSN={0044-2968},
     journal={Z. Krist.},
      volume={183},
      number={1-4},
       pages={129–152},
       doi={10.1524/zkri.1988.183.14.129},
}

\bib{konopelchenko_three-dimensional_1998}{article}{
      author={Konopelchenko, B.~G.},
      author={Schief, W.~K.},
       title={Three-dimensional integrable lattices in {{Euclidean}} spaces:
  conjugacy and orthogonality},
        date={1998},
        ISSN={1364-5021},
     journal={R. Soc. Lond. Proc. Ser. A Math. Phys. Eng. Sci.},
      volume={454},
      number={1980},
       pages={3075\ndash 3104},
       doi={10.1098/rspa.1998.0292},
}

\bib{nutbourne_differential_1988}{book}{
      author={Nutbourne, Anthony~W.},
      author={Martin, Ralph~R.},
       title={Differential geometry applied to curve and surface design.
  {{Vol}}. 1},
   publisher={{Ellis Horwood Ltd.}},
     address={Chichester},
        date={1988},
        ISBN={978-0-7458-0140-7},
}

\bib{pottmann_geometry_2007}{article}{
      author={Pottmann, Helmut},
      author={Liu, Yang},
      author={Wallner, Johannes},
      author={Bobenko, Alexander~I.},
      author={Wang, Wenping},
       title={Geometry of multi-layer freeform structures for architecture},
        date={2007},
     journal={ACM Trans. on Graph. (TOG)},
      volume={26},
      number={3},
       pages={65\ndash 1\ndash 65\ndash 11},
       doi={10.1145/1276377.1276458},
}

\bib{rossman_minimal_1995}{article}{
      author={Rossman, Wayne},
       title={Minimal surfaces in $\bf {R}^3$ with dihedral symmetry},
        date={1995},
        ISSN={0040-8735},
     journal={Tohoku Math. J. (2)},
      volume={47},
      number={1},
       pages={31–54},
       doi={10.2748/tmj/1178225634},
}

\bib{rossman_irreducible_1997}{article}{
      author={Rossman, Wayne},
      author={Umehara, Masaaki},
      author={Yamada, Kotaro},
       title={Irreducible constant mean curvature $1$ surfaces in hyperbolic
  space with positive genus},
        date={1997},
        ISSN={0040-8735},
     journal={Tohoku Math. J. (2)},
      volume={49},
      number={4},
       pages={449–484},
       doi={10.2748/tmj/1178225055},
}

\bib{rossman_discrete_2018}{incollection}{
      author={Rossman, Wayne}, 
      author={Yasumoto, Masashi},
       title={Discrete linear {{Weingarten}} surfaces with singularities in
  {{Riemannian}} and {{Lorentzian}} spaceforms},
        date={2018},
   booktitle={Singularities in generic geometry},
      editor={Izumiya, Shyuichi},
      editor={Ishikawa, Goo},
      editor={Yamamoto, Minoru},
      editor={Saji, Kentaro},
      editor={Yamamoto, Takahiro},
      editor={Takahashi, Masatomo},
      series={Adv. Stud. Pure Math.},
      volume={78},
   publisher={{Math. Soc. Japan}},
     address={Tokyo},
       pages={383\ndash 410},
       doi={10.2969/aspm/07810383},
}

\bib{sauer_parallelogrammgitter_1950}{article}{
      author={Sauer, Robert},
       title={Parallelogrammgitter als {{Modelle}} pseudosph\"arischer
  {{Fl\"achen}}},
        date={1950},
        ISSN={0025-5874},
     journal={Math. Z.},
      volume={52},
       pages={611\ndash 622},
       doi={10.1007/BF02230715},
}

\bib{sauer_differenzengeometrie_1970}{book}{
      author={Sauer, Robert},
       title={Differenzengeometrie},
   publisher={{Springer-Verlag}},
     address={Berlin-New York},
        date={1970},
}

\bib{sauer_uber_1931}{article}{
      author={Sauer, Robert},
      author={Graf, H.},
       title={{\"Uber} {{Fl\"achenverbiegung}} in {{Analogie}} zur
  {{Verknickung}} offener {{Facettenflache}}},
        date={1931},
        ISSN={0025-5831},
     journal={Math. Ann.},
      volume={105},
      number={1},
       pages={499\ndash 535},
       doi={10.1007/BF01455828},
}

\bib{schief_unification_2003}{article}{
      author={Schief, W.~K.},
       title={On the unification of classical and novel integrable surfaces.
  {{II}}. {{Difference}} geometry},
        date={2003},
        ISSN={1364-5021},
     journal={R. Soc. Lond. Proc. Ser. A Math. Phys. Eng. Sci.},
      volume={459},
      number={2030},
       pages={373\ndash 391},
       doi={10.1098/rspa.2002.1008},
}

\bib{schoen_infinite_1970}{techreport}{
      author={Schoen, A.~H.},
       title={Infinite periodic minimal surfaces without self-intersections},
        type={Technical {{Report}}},
 institution={{NASA Electronics Research Center}},
     address={Cambridge, MA},
        date={1970},
      number={NASA-TN-D-5541, C-98},
}

\bib{wunderlich_zur_1951}{article}{
      author={Wunderlich, Walter},
       title={Zur {{Differenzengeometrie}} der {{Fl\"achen}} konstanter
  negativer {{Kr\"ummung}}},
        date={1951},
     journal={\"Osterreich. Akad. Wiss. Math.-Nat. Kl. S.-B. IIa.},
      volume={160},
       pages={39\ndash 77},
}

\bib{xu_symmetric_1995}{article}{
      author={Xu, Youyu},
       title={Symmetric minimal surfaces in $\bf {R}^3$},
        date={1995},
        ISSN={0030-8730},
     journal={Pacific J. Math.},
      volume={171},
      number={1},
       pages={275–296},
       doi={10.2140/pjm.1995.171.275},
}

\end{biblist}
\end{bibdiv}

\end{document}